\documentclass{amsart}
\usepackage{graphicx}
\usepackage[all]{xy}
\usepackage[ansinew]{inputenc}
\usepackage{amsmath}
\usepackage{amscd}
\usepackage{amssymb}
\usepackage{latexsym}
\usepackage[active]{srcltx}
\usepackage{mathrsfs}
\usepackage{color}
\usepackage{amsthm}
\usepackage{mathrsfs}
\usepackage[numbers]{natbib}
\usepackage[colorlinks,citecolor=blue,urlcolor=blue]{hyperref}
\usepackage{tgtermes}
\usepackage{mathptmx}
\usepackage[scaled=.92]{helvet}

\newtheorem{thm}{Theorem}[section]

\theoremstyle{definition}
\newtheorem{cor}[thm]{Corollary}
\newtheorem{lem}[thm]{Lemma}
\newtheorem{prop}[thm]{Proposition}
\newtheorem{defn}[thm]{Definition}

\newtheorem{algo}{Algorithm}

\newtheorem*{thmA}{Theorem A}
\newtheorem*{thmB}{Theorem B}

\numberwithin{equation}{section}

\newcommand{\N}{\mathbb{N}}
\newcommand{\Z}{\mathbb{Z}}

\newcommand{\Cal}{\mathcal}

\def \<{\langle}
\def \>{\rangle}

\def \((  {(\!(}
\def \)) {)\!)}

\begin{document}

\title[]
{Ostrowski numeration systems, addition and finite automata}

\begin{abstract}
We present an elementary three pass algorithm for computing addition in Ostrowski numerations systems. When $a$ is quadratic,
addition in the Ostrowski numeration system based on $a$ is recognizable by a finite automaton. We deduce that a subset of $X\subseteq \N^n$ is definable in $(\N,+,V_a)$, where $V_a$ is the function that maps a natural number $x$ to the smallest denominator of a convergent of $a$ that appears in the Ostrowski representation based on $a$ of $x$ with a non-zero coefficient, if and only if the set of Ostrowski representations of elements of $X$ is recognizable by a finite automaton. The decidability of the theory of $(\N,+,V_a)$ follows.
\end{abstract}

\author[P. Hieronymi]{Philipp Hieronymi}
\address
{Department of Mathematics\\University of Illinois at Urbana-Champaign\\1409 West Green Street\\Urbana, IL 61801}
\email{phierony@illinois.edu}
\urladdr{http://www.math.uiuc.edu/\textasciitilde phierony}

\author[A. Terry]{Alonza Terry Jr.}
\address
{Department of Mathematics\\University of Illinois at Urbana-Champaign\\1409 West Green Street\\Urbana, IL 61801}
\email{aterry@illinois.edu}

\thanks{The first author was partially supported by NSF grant DMS-1300402 and by UIUC Campus Research Board award 13086. A version of this paper will appear in the \emph{Notre Dame Journal of Formal Logic.}}
\date{\today}

\maketitle

\section{Introduction}

A \textbf{continued fraction expansion} $[a_0;a_1,\dots,a_k,\dots]$ is an expression of the form
\[
a_0 + \frac{1}{a_1 + \frac{1}{a_2+ \frac{1}{a_3+  \frac{1}{\ddots}}}}
\]
For a real number $a$, we say $[a_0;a_1,\dots,a_k,\dots]$ is the continued fraction expansion of $a$ if $a=[a_0;a_1,\dots,a_k,\cdots]$ and $a_0\in \Z$, $a_i\in \N_{>0}$ for $i>0$.
Let $a$ be a real number with continued fraction expansion $[a_0;a_1,\dots,a_k,\dots]$. In this note we study a numeration system due to Ostrowski \cite{Ost} based on the continued fraction expansion of $a$. Set $q_{-1} := 0$ and $q_{0} := 1$, and for $k\geq 0$,
\begin{equation}\label{equation:intro1}
q_{k+1} := a_{k+1} \cdot q_k + q_{k-1}.
\end{equation}
Then every natural number $N$ can be written uniquely as
\[
N = \sum_{k=0}^{n} b_{k+1} q_{k},
\]
where $b_{k} \in \N$ such that $b_1<a_1$, $b_k \leq a_{k}$ and, if $b_k = a_{k}$, $b_{k-1} = 0$. We say the word $b_n\dots b_1$ is the \textbf{Ostrowski representation} of $N$ based on $a$, and we write $\rho_a(N)$ for this word. For more details on Ostrowski representations, see for example Allouche and Shallit \cite[p.106]{Automatic} or Rockett and  Sz\"usz \cite[Chapter II.4]{RS}. When $a$ is the golden ratio $\phi:=\frac{1+\sqrt{5}}{2}$, the continued fraction expansion of $a$ is $[1;1,\dots]$. In this special case the sequence $(q_k)_{k\in \N}$ is the sequence of Fibonacci numbers. Thus the Ostrowski representation based on the golden ratio is precisely the better known \textbf{Zeckendorf representation} \cite{Zeckendorf}.\newline

In this paper, we will study the following question: given the continued fraction expansion of $a$ and the Ostrowski representation of two natural numbers based on $a$, is there an easy way to compute the Ostrowski representation of their sum? Ahlbach, Usatine, Frougny and Pippenger \cite{EffectiveZeckendorf} give an elegant algorithm to calculate the sum of two natural numbers in Zeckendorf representations. In this paper we generalize their work and present an elementary three pass algorithm for computing the sum of two natural numbers given in Ostrowski representation. To be precise, we show that given the continued fraction expansion of $a$, addition of two $n$-digit numbers in Ostrowski representation based on $a$ can be computed by three linear passes over the input sequence and hence in time $O(n)$. If $a$ is  a quadratic number\footnote{A real number $a$ is \textbf{quadratic} if it is a solution to a quadratic equation with rational coefficients}, we establish that the graph of addition in the Ostrowski numeration system based on $a$ can be recognized by a finite automaton (see Theorem B for a precise statement). When $a$ is the golden ratio, this result is due to Frougny \cite{Frougny}\footnote{In private communication Frougny proved that whenever the continued fraction expansion of a has period 1, the stronger statement that addition in the Ostrowski numeration system associated with $a$ can be obtained by three linear passes, one left-to-right, one right-to-left and one left-to-right, where each of the passes defines a finite sequential transducer.}. \newline

Ostrowski representations arose in number theory and have strong connections to the combinatorics of words (see for example Berth\'e \cite{Berthe}). However, our main motivation for studying Ostrowski representations is their application to decidability and definability questions in mathematical logic. The results in this paper (in particular Theorem B below) play a crucial role in the work of the first author \cite{H-Twosubgroups} on expansions of the real additive group. Here we will present the following application of our work on addition in the Ostrowski numeration system to the study of expansions of Presburger Arithmetic (see Theorem A). \newline

Let $a$ be quadratic. Since the continued fraction expansion of $a$ is periodic, there is a natural number $c:= \max_{k\in \N} a_k$. Let $\Sigma_a= \{0,\dots,c\}$. So $\rho_a(N)$ is a $\Sigma_a$-word. Let $V_a : \N \to \N$ be the function that maps $x \geq 1$ with Ostrowski representation $b_n\dots b_1$ to the least $q_k$ with $b_{k+1} \neq 0$, and $0$ to $1$.

\begin{thmA} Let $a$ be quadratic. A set $X\subseteq \N^n$ is definable in $(\N,+,V_a)$ if and only if $X$ is $a$-recognizable. Hence the theory of $(\N,+,V_a)$ is decidable.
\end{thmA}

We say a set $X \subseteq \N$ is \textbf{$a$-recognizable} if $0^*\rho_a(X)$ is recognizable by a finite automaton, where $0^*\rho_a(X)$ is the set of all $\Sigma_a$-words of the form $0\dots 0 \rho_a(N)$ for some $N\in X$.
The definition of $a$-recognizability for subsets of $\N^n$ is slightly more technical and we postpone it to Section 3. The decidability of the theory of $(\N,+,V_a)$ follows immediately from the first part of the statement of Theorem A and Kleene's theorem (see Khoussainov and Nerode \cite[Theorem 2.7.2]{automata}) that the emptiness problem for finite automata is decidable. Bruy\`{e}re and Hansel \cite[Theorem 16]{Bruyere} establish Theorem A when $a$ is the golden ratio. In fact, they show that Theorem A holds for linear numeration systems whose characteristic polynomial is the minimal polynomial of a Pisot number. A similar result for numeration systems based on $(p^n)_{n\in \N}$, where $p>1$ is an integer, is due to B\"uchi \cite{B60} (for a full proof see Bruy{\`e}re, Hansel, Michaux and and Villemaire \cite{BVHM}). It is known by Shallit \cite{Shallit} and Loraud \cite[Theorem 7]{Loraud} that the set $\N$ is  $a$-recognizable if and only if $a$ is quadratic. So in general the conclusion of Theorem A fails when $a$ is not quadratic. \newline

A few remarks about the proof of Theorem A are in order. The proof that every definable set is $a$-recognizable, is rather straightforward, and we follow a similar argument from Villemaire \cite{Villemaire}. For the other direction, by Hodgson \cite{Hodgson} it is enough to prove that $\N$, the graph of $V_a$ and the graph of $+$ are $a$-recognizable. While it is easy to check the $a$-recognizability of the graph of $V_a$, we have to use our algorithm for addition in Ostrowski numeration systems to show that the graph of $+$ is $a$-recognizable. Thus most of the work towards proving Theorem A goes into showing the following result.

\begin{thmB} Let $a$ is a quadratic. Then $\{ (x,y,z) \in \N^3 \ : \ x+y = z \}$ is $a$-recognizable.
\end{thmB}

We end this introduction with a brief comment about possible applications of Theorem B to the theory of Sturmian words\footnote{When preparing this paper, the authors were completely unaware of the connection between Sturmian words and Ostrowski representations. We would like to thank the anonymous referee to point out this connection.}. Let $a$ be a real number in $[0,1]$. We define
\[
f_{a}(n):=\lfloor (n+1)a \rfloor - \lfloor na \rfloor,
\]
and we denote the infinite $\{0,1\}$-word $f_a(1)f_a(2)\dots$ by $\boldsymbol{f}_{a}$. This word is called the \textbf{Sturmian characteristic word} with slope $a$. If $a$ is a quadratic irrational, the set $\{n \in \N \ : f_a(n) =1\}$ is $a$-recognizable (see \cite[Theorem 9.1.15]{Automatic}). Du, Mousavi, Schaeffer and Shallit \cite{Duetal} use this connection and Theorem B in the case of the golden ratio $\phi$ to prove results about the Fibonacci word (that is the Sturmian characteristic word with slope $\phi -1$). Because of Theorem B the techniques in \cite{Duetal} can be applied to any characteristic Sturmian word whose slope is a quadratic irrational.

\subsection*{Notation} We denote the set of natural numbers by $\{0,1,2,\dots\}$ by $\N$. Definable will always mean definable without parameters. If $\Sigma$ is a finite set, we denote the set of $\Sigma$-words by $\Sigma^*$. If $a \in \Sigma$ and $X\subseteq \Sigma^*$, we denote the set $\{ a\dots aw \ : \ w \in X\}$ of $\Sigma$-words by $a^*X$. If $x \in X^m$ for some set $X$, we write $x_i$ for the $i$-th coordinate of $x$.

\section{Ostrowski addition}

Fix a real number $a$ with continued fraction expansion $[a_0;a_1,\dots,a_k,\dots]$. In this section we present an algorithm to compute the Ostrowski representations based on $a$ of the sum of two natural numbers given in Ostrowski representation based on $a$. Since we will only consider Ostrowski representation based on $a$, we will omit the reference to $a$. In the special case that $a$ is the golden ratio, our algorithm is exactly the one presented in \cite{EffectiveZeckendorf}. Although it is not strictly necessary, the reader might find it useful to read \cite[Section 2]{EffectiveZeckendorf} first.\\

\noindent Let $M,N \in \N$ and let $x_n\dots x_1, y_n\dots y_1$ be the Ostrowski representations of $M$ and $N$. We will describe an algorithm that given the continued fraction expansion of $a$ calculates the Ostrowski representation of $M+N$. Let $s$ be the word $s_{n+1}s_n\dots s_1$ given by
\[
s_i := x_i + y_i,
\]
for $i=1,\dots,n$ and $s_{n+1}:=0$. For ease of notation, we set $m:=n+1$.\\

\noindent  The algorithm consists of three linear passes over $s$: one left-to-right, one right-to-left and one left-to-right. These three passes will change the word $s$ into a word that is the Ostrowski representation of $M+N$. The first pass converts $s$ into a word whose digit at position $k$ is smaller or equal to $a_k$. The idea how to achieve this, is as follows. We will argue (see Lemma \ref{lem:smallfirst}) that whenever the digit at position $k$ is larger or equal to $a_k$, then the preceding digit has to be less than $a_{k+1}$. Using \eqref{eq:1} we can then decrease the digit at position $k$ by $a_k$, without increasing the one at position $k+1$ above $a_{k+1}$, and without changing the value the word represents. The resulting word might not yet be an Ostrowski representation of $M+N$, because the digit at position $k$ may be $a_k$ and not followed by $0$. With the second and third pass we eliminate all such occurrences.\\

 \noindent The first step is an algorithm that makes a left-to-right pass over the sequence $s_m\dots s_1$ starting at $m$. That means that it starts with the most significant digit, in this case $s_m$, and works its way down to the least significant digit $s_1$. The algorithm can best be described in terms of a moving window of width four. At each step, we only consider the entries in this window. After any possible changes are performed, the window moves one position to the right. When the window reaches the last four digits, the changes are carried out as usual. Afterwards, one final operation is performed on the last three digits. The precise algorithm is as follows. Given $s=s_m\dots s_1$, we will recursively define for every $k \in \N$ with $3 \leq k \leq m+1$, a word
\[
z_k:= z_{k,m}z_{k,m-1}\dots z_{k,2}z_{k,1}.
\]

\begin{algo} Let $k=m+1$. Then set
\[
z_{m+1} := s_m\dots s_1.
\]

\noindent Let $k\in \N$ with $4 \leq k < m+1$. We now define $z_k= z_{k,m}z_{k,m-1}\dots z_{k,2}z_{k_1}$:

\begin{itemize}
\item for $i\notin \{k,k-1,k-2,k-3\}$, we set $z_{k,i} = z_{k+1,i}$,
\item the subword $z_{k,k}z_{k,k-1}z_{k,k-2}z_{k,k-3}$ is determined as follows: \\

\begin{enumerate}
\item[(A1)] if $z_{k+1,k}<a_k, z_{k+1,k-1} > a_{k-1}$ and $z_{k+1,k-2}=0$,
\[
z_{k,k}z_{k,k-1}z_{k,k-2}z_{k,k-3}=(z_{k+1,k}+1)(z_{k+1,k-1}-(a_{k-1}+1))(a_{k-2}-1)(z_{k+1,k-3}+1)
\]
\item[(A2)] if $z_{k+1,k}<a_k, a_{k-1}\leq z_{k+1,k-1} \leq 2a_{k-1}$ and $z_{k+1,k-2}>0$,
\[
z_{k,k}z_{k,k-1}z_{k,k-2}z_{k,k-3}=(z_{k+1,k}+1)(z_{k+1,k-1}-a_{k-1})(z_{k+1,k-2}-1)(z_{k+1,k-3})
\]
\item[(A3)] otherwise,
\[
z_{k,k}z_{k,k-1}z_{k,k-2}z_{k,k-3}=z_{k+1,k}z_{k+1,k-1}z_{k+1,k-2}z_{k+1,k-3}.
\]
\end{enumerate}
\end{itemize}

\vspace{0.5cm}
Let $k=3$. We now define $z_3=z_{3,m}\dots z_{3,1}$:
\begin{itemize}
\item for $i\notin\{1,2,3\}$, we set $z_{3,l} = z_{4,l}$,
\item the subword $z_{3,3}z_{3,2}z_{3,1}$ is determined as follows: \\
\begin{enumerate}
\item[(B1)] if $z_{4,3} < a_3$, $z_{4,2} > a_2$ and $z_{4,1}=0$,
\[
z_{3,3}z_{3,2}z_{3,1}=(z_{4,3}+1)(z_{4,2}-(a_2+1))(a_1-1),
\]
\item[(B2)] if $z_{4,3} < a_3$, $z_{4,2} \geq a_2$ and $a_1 \geq z_{4,1}>0$,
\[
z_{3,3}z_{3,2}z_{3,1}=(z_{4,3}+1)(z_{4,2}-a_2)(z_{4,1}-1),
\]
\item[(B3)]  if $z_{4,3} < a_3$, $z_{4,2} \geq a_2$ and $z_{4,1}>a_1$,
\[
z_{3,3}z_{3,2}z_{3,1}=(z_{4,3}+1)(z_{4,2}-a_2+1)(z_{4,1}-a_1-1),
\]
\item[(B4)]  if $z_{4,2} < a_2$ and $z_{4,1}\geq a_1$,
\[
z_{3,3}z_{3,2}z_{3,1}= z_{4,3}(z_{4,2}+1)(z_{4,1}-a_1).
\]
\item[(B5)] otherwise,
\[
z_{3,3}z_{3,2}z_{3,1}=z_{4,3}z_{4,2}z_{4,1}.
\]
\end{enumerate}
\end{itemize}
\end{algo}

\noindent When we speak of \textbf{the entry at position $l$ after step $k$}, we mean $z_{k,l}$. When $z_{k+1,l} \neq z_{k,l}$, we say that at step $k$ the entry in position $l$ was changed. It follows immediately from the algorithm that the only entries changed at step $k$, are in position $k,k-1,k-2$ or $k-3$.\\

\noindent The goal of Algorithm 1 is to produce a word whose entry at position $k$ is smaller or equal to $a_k$, and which represents the same value as $s$. The following two Propositions make this statement precise.

\begin{prop} Algorithm 1 leaves the value represented unchanged. That is, for every $k \in \N$ with $3\leq k \leq m+1$
\[
\sum_{i=0}^{m} z_{k,i+1}q_i = \sum_{i=0}^{m} s_{i+1}q_i.
\]
\end{prop}
\begin{proof} It follows immediately from the recursive definition of the $q_i$'s (see \eqref{equation:intro1}) that each rule of Algorithm 1 leaves the value represented unchanged. Induction on $k$ gives the statement of the Proposition.
\end{proof}

\begin{prop}\label{prop:step1} For $k>1$, $z_{3,k} \leq a_k$ and  $z_{3,1} \leq a_1-1$.
\end{prop}

\noindent We will prove the following two lemmas first.

\begin{lem}\label{lem:aftertwo} Let $k\in \N$ and $k\geq 3$. Then
\begin{itemize}
\item[(i)] If $z_{k+1,k-1} = 2a_{k-1}+1$, then $z_{k+1,k-2}=0$.
\item[(ii)] If $z_{k+1,k-1} = 2a_{k-1}$, then $z_{k+1,k-2}\leq a_{k-2}$.
\end{itemize}
\end{lem}
\begin{proof} For (i), let $z_{k+1,k-1} = 2a_{k-1}+1$. It follows immediately from the rules of the algorithm that $z_{k+2,k-1}=2a_{k-1}+1$ and $z_{m+1,k-1}=2a_{k-1}$. So $x_{k-1}$ and $y_{k-1}$ are both equal to $a_{k-1}$. Hence $x_{k-2}=0, y_{k-2}=0$ and $z_{m+1,k-2}=0$. The first time that the entry in position $k-2$ can be changed, is at step $k+1$, when rule (A1) is applied. However, since $z_{k+2,k-1} = 2a_{k-1}+1$, rule (A1) was not applied at step $k+1$. Thus $z_{k+1,k-2}=z_{m+1,k-2}=0$.\\

\noindent For (ii), let $z_{k+1,k-1} = 2a_{k-1}$. If $x_{k-1}=y_{k-1}=a_{k-1}$, we argue as before to get $z_{k+1,k-2}=0$. Suppose that either $x_{k-1}\neq a_{k-1}$ or $y_{k-1} \neq a_{k-1}$. Because $z_{k+1,k-1} = 2a_{k-1}$, we get that $x_{k-1} + y_{k-1}=2a_{k-1}-1$, and  that the entry in position $k-1$ had to be increased by $1$ at step $k+2$. Hence either $x_{k-1} = a_{k-1}$ or $y_{k-1}=a_{k-1}$. By the definition of Ostrowski representations, $x_{k-2} + y_{k-2} \leq a_{k-2}$. Thus $z_{k+2,k-2} \leq a_{k-2}$. Since the entry in position $k-1$ was increased by $1$ at step $k+2$, $z_{k+2,k}=a_k-1$. Thus no change is made at step $k+1$. It follows that $z_{k+1,k-2} = x_{k-2} + y_{k-2} \leq a_{k-2}$.
\end{proof}

\begin{lem}\label{lem:smallfirst} Let $k \in \N$ and $3 \leq k\leq m$.
\begin{itemize}
\item[(i$)_k$] If $z_{k+1,k-1}>a_{k-1}$, then $z_{k+1,k} < a_k$.
\item[(ii$)_k$] If $z_{k+1,k-1}=a_{k-1}$ and $z_{k+1,k-2}>0$, then $z_{k+1,k} < a_{k}$.
\end{itemize}
\end{lem}
\begin{proof} We prove the statements by induction on $k$. For $k=m$, both (i$)_{m}$ and (ii$)_{m}$ hold, because $z_{m+1,m} = 0$.
For the induction step, suppose that (i$)_{k+1}$ and (ii$)_{k+1}$ hold. We need to establish (i$)_{k}$ and (ii$)_k$.\\

\noindent We first show (i$)_k$. Suppose $z_{k+1,k-1} > a_{k-1}$. Towards a contradiction, assume that $z_{k+1,k} \geq a_k$. Since $z_{k+1,k-1}>a_{k-1}$ and the algorithm does not increase the entry in position $k-1$ above $a_{k-1}$ at step $k+1$, we have $z_{k+2,k-1}>a_{k-1}$. Because $z_{k+1,k} \geq a_k$ and the algorithm either leaves the entry in position $k$ at step $k+1$ untouched or decreases it by $a_k$ or $a_k+1$, we get that either $z_{k+2,k}=z_{k+1,k}$ or $z_{k+2,k} \in \{2a_k,2a_k+1\}$. We handle these cases separately.\\

\noindent Suppose $z_{k+2,k} \in \{2a_k,2a_k+1\}$. By (i$)_{k+1}$, $z_{k+2,k+1} < a_{k+1}$. It follows from Lemma \ref{lem:aftertwo} that, if $z_{k+2,k}=2a_k$, then $z_{k+2,k-1} \leq a_{k-1}$, and if $z_{k+2,k}=2a_k+1$, then $z_{k+2,k-1}=0$. Since one of the first two rules is applied at step $k+1$, we have that $z_{k+1,k-1} < a_{k-1}$. This contradicts our assumption that $z_{k+1,k-1} > a_{k-1}$.\\

\noindent Now, we suppose that $z_{k+2,k}=z_{k+1,k}$ and $z_{k+2,k}=a_k$. Because $z_{k+2,k-1} > a_{k-1}$, we get $z_{k+2,k+1} < a_{k+1}$ by (ii$)_{k+1}$. Hence $z_{k+1,k} = z_{k+2,k} - a_k$ by rule (A2). This contradicts $z_{k+1,k} = z_{k+2,k}$. \\

\noindent Finally, assume that $z_{k+2,k}=z_{k+1,k}$ and $z_{k+2,k}>a_k$. By (i$)_{k+1}$, $z_{k+2,k+1} < a_{k+1}$.  Since $z_{k+2,k-1} > a_{k-1}$, we have $z_{k+2,k+1}< 2a_{k+1}$ by Lemma \ref{lem:aftertwo}.  Applying rule (A2) gives $z_{k+1,k}=z_{k+2,k}-a_{k}$. As before, this is a contradiction.\\

\noindent We now prove (ii$)_k$. Let  $z_{k+1,k-1}=a_{k-1}$ and $z_{k+1,k-2}>0$. Suppose towards a contradiction that $z_{k+1,k} \geq a_{k}$. Then $z_{k+2,k} \geq a_{k}$, because the algorithm never increases the entry at position $k$ at step $k+1$. Since $z_{k+1,k-1}=a_{k-1}$, either $z_{k+2,k-1} = a_{k-1}+1$ (in this case rule (A2) was applied) or $z_{k+2,k-1} =a_{k-1}$ (in this case rule (A3) was applied). In both cases, $z_{k+2,k+1}< a_{k+1}$ by  (i$)_{k+1}$ and (ii$)_{k+1}$. Since $z_{k+2,k-1} > 0$, $z_{k+2,k}\leq 2a_k$ by Lemma \ref{lem:aftertwo}(i). Hence rule (A2) was applied at step $k+1$, and $z_{k+2,k-1}=a_{k-1}+1$. By Lemma \ref{lem:aftertwo}(ii), $z_{k+2,k} < 2a_k$.
Thus $z_{k+1,k} = z_{k+2,k} - a_k < a_k$, a contradiction.
\end{proof}

\begin{proof}[Proof of Proposition \ref{prop:step1}]  Suppose $k\geq 3$. Because the entry at position $k$ is not changed after step $k$, it is enough to show that $z_{k,k} \leq a_k$. We have to consider four different cases depending on the value of $z_{k+2,k}$.\\

 \noindent First, consider the case that $z_{k+2,k}< a_k$. Since the algorithm does not increase the entry in position $k$ at step $k+1$, $z_{k+1,k}<a_k$. Thus $z_{k,k} \leq z_{k+1,k} + 1 \leq a_k$. \\

 \noindent Suppose $z_{k+2,k} = a_k$ and $z_{k+2,k-1} > 0$. By Lemma \ref{lem:smallfirst}(ii), $z_{k+2,k+1} < a_{k+1}$. By rule (A2), $z_{k+1,k} = 0$. Hence $z_{k,k}\leq 1 \leq a_k$.\\

 \noindent Suppose $z_{k+2,k} = a_k$ and $z_{k+2,k-1} = 0$. Then no change is made at step $k+1$. Thus $z_{k+1,k}=a_k$ and $z_{k+1,k-1}=0$. Since no change is made at step $k$ as well, $z_{k,k} = a_k$. \\

 \noindent Finally, consider $z_{k+2,k}> a_k$. By Lemma \ref{lem:smallfirst}(i), $z_{k+2,k+1} < a_{k+1}$ . Hence either rule (A1) or rule (A2) is applied. We get that $z_{k+1,k} \leq  a_k$. If $z_{k+1,k}=a_k$, then $z_{k,k}=a_k$. If $z_{k+1,k}<a_k$, then $z_{k,k} \leq z_{k+1,k} + 1 \leq a_k$.\\

\noindent Now suppose that $k< 3$. We have to show that $z_{3,k} \leq a_k$. We do so by considering several different cases depending on the values of $z_{4,2}$ and $z_{4,1}$. By Lemma \ref{lem:smallfirst}, if $z_{4,2}>a_2$, or, if $z_{4,2}=a_2$ and $z_{4,1}>0$, then $z_{4,3}<a_3$. If $z_{4,2}=a_2$ and $z_{4,1}=0$, then no changes was made.\\

\noindent Suppose that $z_{4,2}=2a_2+1$. By Lemma \ref{lem:aftertwo}, $z_{4,1}=0$ . By rule (B1), $z_{3,2} = a_2$, $z_{3,1} = a_1 -1$ and $z_{3,3} =z_{4,3}+1 \leq a_3$.\\

\noindent Now suppose that $z_{4,2} = 2a_2$. We get $z_{4,1} \leq a_1$ from Lemma \ref{lem:aftertwo}. Then either rule (B1) or  rule (B2) was applied. In both cases we get that $z_{3,2}=a_2$, $z_{3,1} = z_{4,1} - 1 \leq a_1 - 1$ and $z_{3,3} =z_{4,3}+1 \leq a_3$.\\

\noindent Consider that $a_2 \leq z_{4,2} < 2a_2$ and $z_{4,1}>0$. Here either rule (B2) or rule (B3) was used. Then $z_{3,2}\leq a_2$, $z_{3,1} \leq a_1-1$ and $z_{3,3} =z_{4,3}+1 \leq a_3$.\\

\noindent The last case we have to consider is $z_{4,2}<a_2$. Depending on whether $z_{4,1}\geq a_1$, we applied either rule (B4) or rule (B5). Since $z_{4,1} \leq 2a_1-1$, we get $z_{3,1} \leq a_1 -1$ and $z_{3,2} \leq z_{3,2} +1 \leq a_2$ in both cases.
\end{proof}

\noindent We will now describe the second step towards determining the Ostrowski representation of $M+N$. This second algorithm will be a right-to-left pass over $z_3$. Given the word $z_{3,m}z_{3,m-1}\dots z_{3,2}z_{3,1}$, we will recursively generate a word
\[
w_k= w_{k,m+1}w_{k,m}\dots w_{k,2}w_{k,1}
\]
 for each $k\in N$ with $k \in \N$ with $2\leq k \leq m+1$. At each step only elements in a moving window of length $3$ are changed. Because the algorithm moves right to left, we will start by defining $w_2$, and then recursively define $w_k$ for $k\geq 2$.

\begin{algo}Let $k=2$. Then set
\[
w_{2} := 0z_{3,m}z_{3,m-1}\dots z_{3,2}z_{3,1}.
\]

\noindent Let $k \in \N$ with $2< k \leq m+1$. We now define $w_k=w_{k,m+1}\dots w_{k,1}$:
\begin{itemize}
\item for $i\notin \{k,k-1,k-2\}$, we set $w_{k,i} := w_{k-1,i}$.
\item if $w_{k-1,k} < a_k$, $w_{k-1,k-1} = a_{k-1}$ and $w_{k-1,k-2}>0$, set
\[
w_{k,k}w_{k,k-1}w_{k,k-2} := (w_{k-1,k}+1)0(w_{k-1,k-2}-1),
\]
otherwise
\[
w_{k,k}w_{k,k-1}w_{k,k-2} := w_{k-1,k}w_{k-1,k-1}w_{k-1,k-2}.
\]
\end{itemize}
\end{algo}

\noindent Again it follows immediately from Equation \eqref{equation:intro1} that this algorithm leaves the value represented unchanged:
\[
\sum_{k=0}^{m} w_{m+1,k+1}q_k = \sum_{k=0}^{m} z_{3,k+1}q_k.
\]
By Proposition \ref{prop:step1} and the rules of Algorithm 2, $w_{k,i} \leq a_k$ for every $k=2,\dots,m+1$ and $i=1,\dots, m+2$.

\begin{lem}\label{lem:step2} There is no $k\in \N$ such that
\begin{itemize}
\item $w_{m+1,k} = a_k$
\item $w_{m+1,k-1} < a_{k-1}$,
\item $w_{m+1,k-2} = a_{k-2}$, and
\item $w_{m+1,k-3} > 0$.
\end{itemize}
\end{lem}
\begin{proof} Towards a contradiction, suppose that there is such an $k$. We will first show that $w_{k-2,k-3}>0, w_{k-2,k-2} = a_{k-2}$ and $w_{k-2,k-1} = a_{k-1}$.\\

 \noindent Suppose that $w_{k-2,k-3}=0$. Then the algorithm would not have made any changes at step $k-2$. Thus $w_{k-1,k-3}=0$. Because the entry will not be changed later than step $k-1$, $w_{m+1,k-3}=0$. However, this contradicts $w_{m+1,k-3}>0$. Thus $w_{k-2,k-3}>0$.\\

\noindent Suppose that $w_{k-2,k-2} < a_{k-2}$. Then $w_{k-1,k-2} = w_{k-2,k-2}$. This implies that $w_{k,k-2} < a_{k-2}$ and $w_{m+1,k-2} < a_k$. This a contradiction against our assumption $w_{m+1,k-2} = a_{k-2}$. Hence  $w_{k-2,k-2} = a_{k-2}$. \\

\noindent Now suppose that $w_{k-2,k-1} < a_{k-1}$. Since $w_{k-2,k-2} = a_{k-2}$ and $w_{k-2,k-3}>0$, $w_{k-1,k-2}=0$. Thus $w_{m+1,k-2} = 0$, contradicting $w_{m+1,k-2} = a_{k-2}$. So $w_{k-2,k-1} = a_{k-1}$.\\

\noindent It follows that $w_{k-1,k-1} = w_{k-2,k-1}=a_{k-1}$ and $w_{k-1,k-2} = w_{k-1,k-2}=a_{k-2}$. We will now argue that $w_{k-1,k}<a_k$.\\

 \noindent Suppose towards a contradiction that $w_{k-1,k} = a_k$. Then $w_{k,k} = a_k$ and $w_{k,k-1} = a_{k-1}$. Since $w_{m+1,k-1}<a_{k-1}$, we have $w_{k,k+1} < a_{k+1}$. Thus $w_{k+1,k} = 0$. Hence $w_{m+1,k} = 0$, a contradiction. So $w_{k-1,k} < a_k$.\\

\noindent  We conclude that the entry at position $k-2$ is changed at step $k$. Therefore, $w_{k,k-2} = w_{k-1,k-2} -1= a_{k-2}-1$. So $w_{m+1,k-2} = a_{k-2}-1$. This contradicts our original assumption $w_{m+1,k-2}=a_{k-2}$.
\end{proof}

\noindent The third and final step of our algorithm is a left-to-right pass over $w_{m+1}$. The moving window is again of length $3$ and we use the same rule as in step 2. Given the word $w_{m+1,m+1}\dots w_{m+1,1}$, we will recursively generate a word
\[
v_k:=v_{k,m+2}\dots v_{k,1}
\]
for each $k\in N$ with $k \in \N$ with $3\leq k \leq m+3$. Because the algorithm moves left to right, we will start by defining $w_{m+3}$ and then recursively define $w_k$ for $k\leq m+3$.

\begin{algo} Let $k=m+3$. Then set
\[
v_{m+3} := 0w_{m+1,m+1}\dots w_{m+1,1}.
\]

Let $k \in \N$ with $3\leq k \leq m+2$. We now define $v_k=v_{k,m+2}\dots v_{k,1}$:

\begin{itemize}
\item for $i\notin \{k,k-1,k-2\}$, we set $v_{k,i} := v_{k+1,i}$,
\item if $v_{k+1,k} < a_k$, $v_{k+1,k-1} = a_{k-1}$ and $v_{k+1,k-2}>0$, set
\[
v_{k,k}v_{k,k-1}v_{k,k-2} := (v_{k+1,k}+1)0(v_{k+1,k-2}-1),
\]
otherwise
\[
v_{k,k}v_{k,k-1}v_{k,k-2} := v_{k+1,k}v_{k+1,k-1}v_{k+1,k-2}.
\]
\end{itemize}
\end{algo}
\noindent As before Equation \eqref{equation:intro1} implies that this algorithm leaves the value represented unchanged:
\[
\sum_{k=0}^{m} w_{m+1,k+1}q_k = \sum_{k=0}^{m} v_{3,k+1}q_k.
\]
Moveover, we have $v_{k,i} \leq a_k$ for every $k=3,...,m+3$ and $i=1,\dots, m+2$. We will now show $v_3$ is indeed the Ostrowski representation of $M+N$. It is enough to prove the following Proposition.

\begin{prop}\label{prop:step3} Let $l \geq 3$. Then there is no $k\geq l-1$ such that $v_{l,k} = a_k$ and $v_{l,k-1} > 0$.
\end{prop}

\noindent Before we give the proof of Proposition \ref{prop:step3}, we need one more Lemma.

\begin{lem}\label{lem:step3} Let $l\in \{3,\dots,m+3\}$. Then there is no $k\in \N$ such that
\begin{itemize}
\item $v_{l,k} = a_k$
\item $v_{l,k-1} < a_{k-1}$,
\item $v_{l,k-2} = a_{k-2}$, and
\item $v_{l,k-3} > 0$.
\end{itemize}
\end{lem}
\begin{proof} We prove the Lemma by induction on $l$. By Lemma \ref{lem:step2}, there is no such $k$ for $m+3$. Suppose that the statement holds for $l+1$. We want to show the statement for $l$. Towards a contradiction, suppose that there is a $k$ such that
\begin{equation}\label{eq:1}
v_{l,k} = a_k, v_{l,k-1} < a_{k-1}, v_{l,k-2} = a_{k-2} \hbox{ and } v_{l,k-3}>0.
\end{equation}
By the induction hypothesis, it is enough to check that no change was made at step $l$; that is $v_{l,i} = v_{l+1,i}$ for $i\in \{k,...,k-3\}$. Since the algorithm only modifies the entries at position $l,l+1$ or $l+2$, we can assume that $k \in \{l-2,\dots,l+3\}$. We consider each case separately. \\

\noindent First, suppose $k=l-2$. We get that $v_{l,i}=v_{l+1,i}$ for $i\in \{k-1,k-2,k-3\}$, because they are not in the moving window at step $l$. The only possible change is at position $k$. Since $v_{l,l-2} < v_{l+1,l-2}$ by induction hypothesis, and $v_{l,l-2}=a_{l-2}$, we get $v_{l,k}=v_{l+1,k}$. So no change is made. \\

\noindent Suppose that $k=l-1$. If a change is made at step $l$, then $v_{l,k}=0$. But this contradicts \eqref{eq:1}. Hence no change is made in this case.\\

\noindent Suppose that $k=l$. If a change is made at step $l$, then $v_{l,k-2}= v_{l+1,k-2} - 1 < a_{k-2}$. As before, this contradicts \eqref{eq:1}. Thus no change is made. \\

\noindent Suppose $k=l+1$. If a change is made at step $l$, then $v_{l,k-2}=0$ contradicting \eqref{eq:1}. So no change is made in this case either.\\

 \noindent Suppose $k=l+2$. If a change is made at step $l$, then $v_{l,k-3}=0$. This again contradicts \eqref{eq:1}, and hence no change is made.\\

\noindent Finally suppose $k=l+3$. By induction hypothesis, $v_{l+1,k-3}=0$. Since $v_{l,k-3}>0$, we have $v_{l+1,k-4} =a_{k-4}$ and $v_{l+1,k-5}>0$. Then
\[
v_{l+1,k-2} = a_{k-2}, v_{l+1,k-3} = 0, v_{l+1,k-4} =a_{k-4} \hbox{ and } v_{l+1,k-5}>0.
\]
This contradicts the induction hypothesis.
\end{proof}

\begin{proof}[Proof of Propositon \ref{prop:step3}] We prove this statement by induction on $l$. For $l=m+3$ the statement holds trivially, because $v_{m+3,m+2}=0$. Now suppose that the statement holds for $l+1$, but fails for $l$. Hence there is $k \geq l-1$ such that $v_{l,k} = a_k$ and $v_{l,k-1} > 0$. Since $v_{l+1,i}=v_{l,i}$ for $i>l$, we have $k\leq l+1$. We now consider the three remaining cases $k=l+1$, $k=l$ and $k=l-1$ individually.\\

\noindent If $k=l+1$, then $v_{l+1,k}=a_{l+1,k}$. By the induction hypothesis, $v_{l+1,k-1}=0$. But in order for $v_{l,k-1}>0$ to hold, we must have $v_{l+1,k-2}=a_{k-2}$ and $v_{l+1,k-3}>0$. This contradicts Lemma \ref{lem:step3}.\\

\noindent If $k=l$, then either $v_{l+1,k}=a_k$ or $v_{l+1,k} = a_k-1$. Suppose that $v_{l+1,k} = a_{k}-1$. Then $v_{l+1,k-1} = a_k$ and $v_{l+1,k-2}>0$. This implies $v_{l,k-1}=0$, which contradicts $v_{l,k-1}>0$. Suppose that $v_{l+1,k} = a_k$. By induction hypothesis, $v_{l+1,k-1}=0$. But then no change is made at step $l$, and hence $v_{l,k-1}=0$. A contradiction against $v_{l,k-1}>0$.\\

\noindent If $k=l-1$,  then no change is made at step $l$, since $v_{l,l-1}=a_{l-1}$. Hence $v_{l+1,l-1}=v_{l,l-1}=a_{l-1}$ and $v_{l+1,l-2} = v_{l,l-2} > 0$. Since no change was made at step $l$, we get that $v_{l+1,l}=a_{l}$. This contradicts the induction hypothesis.
\end{proof}

\begin{cor}\label{cor:ostadd} The word $v_{3,m+2}\dots v_{3,1}$ is the Ostrowski representation of $M+N$.
\end{cor}

\section{Proof of Theorem A}

In this section we will prove Theorem A. Let $a$ be a quadratic irrational number. Let $[a_0;a_1,\dots,a_n,\dots]$ be its continued fraction expansion. Since the continued fraction expansion of $a$ is periodic, it is of the form
\[
[a_0;a_1,\dots,a_{\xi-1},\overline{a_{\xi},\dots, a_{\nu}}],
\]
where $\nu-\xi$ is the length of the repeating block and the repeating block starts at $\xi$. We can choose $\xi$ and $\nu$ such that $\xi>4$ and $\nu - \xi\geq 3$.\footnote{It might be the case that neither $\xi$ nor $\nu$ are minimal, but this will be irrelevant here.} Set $\mu := \max_i a_i$. Set $m:=2\mu + 1$. Set $\Sigma_a:=\{0,\dots,m\}.$ \newline

\noindent We first remind the reader of the definitions of finite automata and recognizability. For more details, we refer the reader to \cite{automata}. Let $\Sigma$ be a finite set. We denote by $\Sigma^*$ the set of words of finite length on $\Sigma$.

\begin{defn} A \textbf{nondeterministic finite automaton} $\Cal A$ over $\Sigma$ is a quadruple $(S,I,T,F)$, where $S$ is a finite non-empty set, called the set of states of $\Cal A$, $I$ is a subset of $S$, called the set of initial states, $T\subseteq S \times \Sigma \times S$ is a non-empty set, called the transition table of $\Cal A$ and $F$ is a subset of $S$, called the set of final states of $\Cal A$. An automaton $\Cal A=(S,I,T,F)$ is \textbf{deterministic} if $I$ contains exactly one element, and for every $s\in S$ and $w \in \Sigma^*$ there is exactly one $s' \in S$ such that $(s,w,s')\in T$.
We say that an automaton $\Cal A$ on $\Sigma$ \textbf{accepts} a word $w=w_n\dots w_1 \in \Sigma^*$ if there is a sequence $s_n,\dots, s_1,s_0 \in S$ such that $s_n \in I$, $s_0 \in F$ and for $i=1,\dots,n$, $(s_i,w_i,s_{i-1})\in T$. A subset $L\subseteq \Sigma^*$ is \textbf{recognized} by $\Cal A$ if $L$ is the set of $\Sigma$-words that are accepted by $\Cal A$. We say that $L\subseteq \Sigma^*$ is \textbf{recognizable} if $L$ is recognized by some deterministic finite automaton.
\end{defn}

\noindent It is well known (see \cite[Theorem 2.3.3]{automata}) that a set is recognizable if it is recognized by some \emph{nondeterministic} finite automaton.\newline

\noindent Let $\Sigma$ be a set containing $0$. Let $z=(z_1,\dots,z_n) \in (\Sigma^*)^n$ and let $m$ be the maximal length of $z_1,\dots,z_n$. We add to each $z_i$ the necessary number of $0$'s to get a word $z_i'$ of length $m$. The \textbf{convolution}\footnote{Here we followed the presentation in \cite{Villemaire}. For a general definition of convolution see \cite{automata}.} of $z$ is defined as the word $z_1 * \dots * z_n \in (\Sigma^n)^*$ whose $i$-th letter is the element of $\Sigma^n$ consisting of the $i$-th letters of  $z_1', \dots, z_n'$.

\begin{defn} A subset $X\subset (\Sigma^*)^n$ is \textbf{$\Sigma$-recognizable} if the set
\[
\{z_1*\dots *z_n \ : \ (z_1,\dots,z_n) \in X\}
\]
is $\Sigma^n$-recognizable.
\end{defn}

\noindent We remind the reader that every natural number $N$ can be written as $N = \sum_{k=0}^{n} b_{k+1} q_{k}$, where $b_{k} \in \N$ such that $b_1<a_1$, $b_k \leq a_{k}$ and, if $b_k = a_{k}$, $b_{k-1} = 0$, and that we denoted the $\Sigma_a$-word $b_n\dots b_1$ by $\rho_a(N)$.

\begin{defn} Let $X\subseteq \N^n$. We say that $X$ is \textbf{$a$-recognizable} if the set
\[
\{ (0^{l_1}\rho_a(N_1),\dots,0^{l_n}\rho_a(N_n)) \ : \ (N_1,\dots,N_n) \in X, l_1,\dots,l_n \in \N\}
\]
is $\Sigma_a$-recognizable.
\end{defn}

\noindent In this section we will prove that a subset $X\subseteq \N^n$ is $a$-recognizable if and only if $X$ is definable in $(\N,+,V_a)$.

\subsection*{Recognizability implies definability} We will first show that whenever a set $X \subseteq \N^n$ is $a$-recognizable, then $X$ is definable in $(\N,+,V_a)$. The proof here is an adjusted version of the proofs in Villemaire \cite{Villemaire} and \cite{Bruyere}.\newline

\noindent First note that $<$ is definable in $(\N,+,V_a)$ and so is $V_a(\N) = \{ q_k \ : \ k \in \N\}$. For convenience, we write $I$ for $V_a(\N)$. We denote the successor function on $I$ by $s_I$.

\begin{defn}For $j\in \{1,\dots, m\}$, let $\epsilon_{j} \subseteq I \times \N$ be the set of $(x,y) \in I \times \N$ with
\begin{align*}
\exists z &\in \N \exists t\in \N (z<x \wedge z+jx< s_I(x) \wedge V_a(t) > x \wedge V_a(x+t) = x \wedge y = z+jx+t) \\
 &\vee \exists z \in \N (z<x \wedge y < s_I(x)\wedge y = z+jx).
\end{align*}
Let $\epsilon_{0} \subseteq I \times \N$ be the set of $(x,y) \in I \times \N$ with $\bigwedge_{j=1}^m \neg \epsilon_j(x,y)$.
\end{defn}

\noindent This definition is inspired by \cite[Lemma 2.3]{Villemaire}. Obviously, $\epsilon_j$ is definable in $(\N,+,V_a)$. Because of the greediness of the Ostrowski representation, $\epsilon_j(x,y)$ holds iff $x=q_k$ for some $k\in \N$ and the coefficient
of $q_k$ in the Ostrowski representation of $y$ is $j$. We directly get the following Lemma.

\begin{lem}\label{lem:rl1} Let $l,n \in \N$ and let $\sum_{k} b_{k+1} q_k$ be the Ostrowski representation of $n$. Then
$b_{l+1} = j$ iff $\epsilon_j(q_l,n)$.
\end{lem}

\begin{defn} Let $I_e$ be the set of all $y \in I$ with
\[
\exists z \in \N \ \epsilon_1(1,z) \wedge \epsilon_1(y,z) \wedge \forall x \in I \big(\epsilon_1(x,z) \leftrightarrow \neg \epsilon_1(s_I(x),z)\big),
\]
and let $I_o$ be the set of all $y \in I$ with
\[
\exists z \in \N \ (\neg \epsilon_1(1,z)) \wedge \epsilon_1(y,z) \wedge \forall x \in I \big(\epsilon_1(x,z) \leftrightarrow \neg \epsilon_1(s_I(x),z)\big).
\]
\end{defn}
\noindent Obviously both $I_e$ and $I_o$ are definable in $(\N,+,V_a)$, $I=I_e \cup I_o$, and since $q_0=1$,
\[
I_e = \{ q_k \ : \ k \hbox{ even } \} \hbox{ and } I_o = \{ q_k \ : \ k \hbox{ odd } \}.
\]
\begin{defn} Let $U_{e} \subseteq \N$ be the set of all $y \in \N$ with
\[
\forall z \in I_o \  \epsilon_0(z,y) \wedge \forall z \in I_e \ (\epsilon_0(z,y) \vee \epsilon_1(z,y)),
\]
and $U_{o} \subseteq \N$ be the set of all $y \in \N$ with
\[
\forall z \in I_e \ \epsilon_0(z,y) \wedge \forall z \in I_o \ (\epsilon_0(z,y) \vee \epsilon_1(z,y)).
\]
\end{defn}

\noindent Again it is easy to see that $U_e$ and $U_o$ are definable in $(\N,+,V_a)$. We get the following Lemma from Lemma \ref{lem:rl1}.

\begin{lem}\label{lem:rlost} Let $n \in \N$ and let $\sum_{k} b_{k+1} q_k$ be the Ostrowski representation of $n$. Then
\begin{itemize}
\item [(i)] $n \in U_e$ if and only if for all even $k$ $b_{k+1} \leq 1$, and for all odd $k$ $b_{k+1}=0$,
\item [(ii)] $n \in U_o$ if and only if for all odd $k$ $b_{k+1} \leq 1$, and for all even $k$ $b_{k+1}=0$.
\end{itemize}
\end{lem}

\begin{defn} Let $\epsilon \subseteq I \times (U_e \times U_o)$ be the set of all $(x,(y_1,y_2))$ with
\[
(x \in I_e \rightarrow \epsilon_1(x,y_1)) \wedge (x \in I_o \rightarrow \epsilon_1(x,y_2)).
\]
\end{defn}

\begin{thm}\label{thm:recidef} Let $X \subseteq \N^n$ be $a$-recognizable. Then $X$ is definable in $(\N,+,V_a)$.
\end{thm}
\begin{proof}  Let $X \subseteq \N^n$ be $a$-recognizable by a finite automaton $\Cal A = (S,I,T,F)$. Without loss generality we can assume that the set of states $S$ is $\{1,\dots,t\}$ for some $t\in \N$, and $I=\{1\}$.
Let $\varphi$ be the formula defining the following subset $Z$ of $U^t$:
\[
\{ (u_1,\dots,u_t) \in U^t \ : \forall q \in I \ \bigwedge_{i=1}^{t} \big(\epsilon(q,u_i) \rightarrow \bigwedge_{j=1, j\neq i}^t \neg \epsilon(q,u_j)\big)\}.
\]
So $Z$ is the set of tuples $(u_1,\dots,u_t) \in U^t$ such that for $q \in I$ there is at most one $i\in \{1,\dots,t\}$ such that $\epsilon(q,u_i)$. Note that $x \in X$ if there is a run $s_{1}\dots s_{m}$ of $\Cal A$ on the word given by the Ostrowski representation of the coordinates of $x$ such that $s_{1} = 1$ and $s_{m} \in F$. The idea now is to code such a run as an element of $Z$. To be precise, a tuple $(u_1,\dots,u_t) \in Z$ will code a run $s_{1}\dots s_{m}$ if for each $q_i \in I$, $s_i$ is the unique element $k$ of $\{1,\dots,t\}$ such that $\epsilon(q_i,u_{k})$. Thus $x=(x_1,\dots,x_n) \in X$ if and only if $x$ satisfies the following formula in $(\N,+,V_a)$:
\begin{align*}
\exists u_1,\dots,u_t \in U \ \exists q \in I \ \varphi(u_1,\dots,u_t) \wedge \epsilon(1,u_1) \wedge \bigvee_{l\in F} \epsilon(q,u_l)\\
\wedge \bigwedge_{(l,(\rho_1,\dots,\rho_n),k)\in T} \forall z \in I  \Big((z > q) \rightarrow \bigwedge_{i=1}^n \bigwedge_{j=1}^{m} \neg \epsilon_{j}(z,x_i)\Big)\\
\wedge  \Big[\big(z \leq q \wedge \epsilon(z,u_l) \wedge \bigwedge_{i=1}^n \epsilon_{\rho_i}(z,x_i)\big)\rightarrow \epsilon(s_I(z),u_k)\Big].
\end{align*}
\end{proof}

\subsection*{Definability implies recognizability} We will prove that if a subset $X\subseteq \N^n$ is definable in $(\N,+,V_a)$, then it is $a$-recognizable. By \cite{Hodgson} it is suffices to show that the set $\N$ and the relations $\{(x,y) \in \N^2 \ : \ x=y\}$, $\{(x,y,z) \in \N^3 \ : \ x+y = z\}$ and $\{(x,y) \in \N^2 \ : \ V_a(x)=y\}$ are all $a$-recognizable. It is well known that $\N$ is $a$-recognizable (see for example \cite[Theorem 8]{Shallit}), and using that knowledge it is easy to check that $\{(x,y) \in \N^2 \ : \ x=y\}$ and $\{(x,y) \in \N^2 \ : \ V_a(x)=y\}$ are $a$-recognizable. We are now going to show that $\{(x,y,z) \in \N^3 \ : \ x+y = z\}$ is $a$-recognizable.\newline

\noindent By the work in the previous section, we have an algorithm to compute addition in Ostrowski representation based on $a$.  This algorithm consists of four steps, and we will now show that each of the four steps can be recognized by a finite automaton. Given two words $z=z_n\dots z_1, z'=z_n'\dots z_1' \in \rho_a(\N)$, the first step is to compute the $\Sigma_a$-word $(z_n+z_n')\dots (z_1+z_1')$, which we will denote by $z+z'$. It is straightforward to verify that the set $\{ z * z' * (z+z') \ : \ z,z' \in \rho_a(\N) \}$ is recognizable by a finite automaton. For $z,z' \in \Sigma_a^*$, we will write $z \rightsquigarrow_i z'$ if Algorithm $i$ produces $z'$ on input $z$. In the following, we will prove that the set $\{ z * z' \ : \ z,z' \in \Sigma_a^*, z\rightsquigarrow_i z'\}$ is recognizable by a finite automaton  for $i=1,2,3$. From these results it is immediate that
\begin{align*}
\{ z * z' * z'' * u_0 * u_1 * u_2 \ &: \ z,z',z'' \in \rho_a(\N), u_0,u_1,u_2 \in \Sigma_a^*,\\
& \  u_0 = z+z', u_0\rightsquigarrow_1 u_1 \rightsquigarrow_2 u_2 \rightsquigarrow_3 z'' \}
\end{align*}
is recognizable by a finite automaton. Since recognizability is preserved under projections (see \cite[Theorem 2.3.9]{automata}), $\{(x,y,z) \in \N^3 \ : \ x+y = z\}$ is $a$-recognizable by Corollary \ref{cor:ostadd}. Thus every set $X \subseteq \N^n$ definable in $(\N,+,V_a)$ is $a$-recognizable.

\subsection*{An automaton for Algorithm 1} We will now construct a non-deterministic automaton $\Cal A_1$ that recognizes the set $\{ z * z' \ : \ z,z' \in \Sigma_a^*, z \rightsquigarrow_1 z'\}$. Before giving the definition of $\Cal A_1$, we need to introduce some notation.  Let $A \subseteq \N_{\leq m}^4\times \N_{\leq m}^4 \times \N_{\leq m}^4$ be the set of tuples $(u,v,w)$ with
\[
w = \left\{
   \begin{array}{ll}
                                           (v_1+1,v_2-(u_2+1),u_3-1,v_4+1), & \hbox{ if $v_1<u_1, v_2 > u_{2}$ and $v_3=0$,} \\
                                              (v_1+1,v_2-u_2,v_3-1,v_4, & \hbox{ if $v_{1}<u_1, u_2\leq v_2 \leq 2u_2$ and $v_3>0$,} \\
                                              (v_1,v_2,v_3,v_4), & \hbox{otherwise.}
                                            \end{array}
                                          \right.
\]
Let $B \subseteq \N_{\leq m}^3\times \N_{\leq m}^3 \times \N_{\leq m}^3$ be the set of tuples $(u,v,w)$ with
\[
w= \left\{
  \begin{array}{ll}
    (v_1+1,v_2-(u_2+1),u_3-1), & \hbox{$v_1 < u_1$, $v_2 > u_2$ and $v_3=0$;} \\
    (v_1+1,v_2-u_2,v_3-1), & \hbox{$v_1 < u_1$, $v_2 \geq u_2$ and $u_1 \geq v_1>0$,;} \\
    (v_1+1,v_2-u_2+1,v_1-u_1-1), & \hbox{$v_{1} < u_1$, $v_{2} \geq u_2$ and $v_1>u_1$;} \\
    (v_1,v_2+1,v_1-u_1), & \hbox{if $v_2 < u_2$ and $v_1\geq u_1$;} \\
    (v_1,v_2,v_3), & \hbox{otherwise.}
  \end{array}
\right.
\]
Note that $A$ corresponds to the rules (A1),(A2) and (A3) of Algorithm 1, while $B$ corresponds to the rules (B1)-(B5) of Algorithm 1. The values of the variable $u$ represent the relevant part of the continued fraction, the values of the variable $v$ are used to code the entries in the moving window before any changes are carried out, and the values of the variable $w$ correspond to the entries in the moving window after the changes are carried out. For $i \in \{4,\dots,\nu\}$ and $l\in \{0,1\}$,
\[
P(i,l) := \left\{
         \begin{array}{ll}
          (a_i,a_{i-1},a_{i-2},a_{\nu}), & \hbox{$i=\xi+2$ and $l=1$;} \\
          (a_i,a_{i-1},a_{\nu},a_{\nu-1}), & \hbox{$i=\xi+1$ and $l=1$;} \\
          (a_i,a_{\nu},a_{\nu-1},a_{\nu-2}), & \hbox{$i=\xi$ and $l=1$;} \\
          (a_i,a_{i-1},a_{i-2},a_{i-3}), & \hbox{otherwise.} \\
         \end{array}
       \right.
\]

\noindent  We first explain informally the construction of $\Cal A_1$. Suppose we take $z=z_l\dots z_1\in \Sigma_a^*$. Now perform Algorithm 1 on $z$, and let the word $z'=z_l'\dots z_1'$ be the output. In order to carry out the operations at step $k$ in Algorithm 1, we needed to know the values of $a_k,a_{k-1},a_{k-2},a_{k-3}$. Because of the periodicity of the continued fraction expansion of $a$, there is $i\leq \nu$ such $a_k=a_i$. Let $l$ be $1$ if $k> \nu$ and $0$ otherwise. Then $P(i,l) = (a_k,a_{k-1},a_{k-2},a_{k-3})$. Hence in order to reconstruct $(a_k,a_{k-1},a_{k-2},a_{k-3}),$ it is enough to save $i$ and whether or not $k\leq \nu$. Moreover, to perform the operations at step $k$ in Algorithm 1, we also used the values of the last three entries in the moving window after the changes in the previous step are carried out, but before the window moves to the right. Let us denote the triple consisting of these entries by $v=(v_1,v_2,v_3) \in \Sigma_a^3$. So before the operations at step $k$ are performed, the values in the moving window are $(v_1,v_2,v_3,z_{k-3})$. Note that at step $k$ in the algorithm, we are reading in $z_{k-3}$, and not $z_k$. However, the value of $z_k'$ is determined at the same step. Indeed, at step $k$ with $k\geq 4$, the entries in the moving window are changed as follows:
\[
(v_1,v_2,v_3,z_{k-3})\mapsto (z_k',v_1',v_2',v_3'),
\]
for a certain triple $(v_1',v_2',v_3') \in \Sigma_a^3$ with $A(P(i,l),v_1,v_2,v_3,z_{k-3},z_k',v_1',v_2',v_3')$. The values in the moving window for step $k-1$ will be $(v'_1,v_2',v_3',z_{k-4})$. Because the value of $z_k'$ is only determined at step $k$, and thus at the same time $z_{k-3}'$ is being read, we are required to store the value of $z_k'$ for three steps.  In order to save this information when moving from state to state, we introduce another triple $(w_1,w_2,w_3) \in \Sigma_a^3$. This triple will always contain the last three digits of $z'$. That means that before step $k$, $(w_1,w_2,w_3)=(z'_{k},z'_{k-1},z'_{k-2})$. We now define the set of states of $\Cal A_1$ as the set of quadruples $(i,l,v,w)$, where $i\leq \nu$, $l \in \{0,1\}$, $v,w\in \Sigma_a^3$. The idea is that in each state of the automaton the pair $(i,l)$ codes the relevant part of the continued fraction expansion, $v$ contains the entries of the moving window, and $w\in \Sigma_a^3$ the values of $z_k'$ that we needed to save. The automaton moves from one of these states to another according to the rules described in Algorithm 1.\newline

\noindent Here is the definition of the automaton $\Cal A_1 =(S_1,I_1,T_1,F_1)$.
\begin{itemize}
\item [1.] The set $S_1$ of states of $\Cal A_1$ is
\begin{align*}
\{ (i,1,v,w) \ : \ \xi \leq i & \leq \nu, v,w \in \Sigma_a^3\}\\
&\cup \{ (i,0,v,w) \ : \ 3 \leq i\leq \nu, v,w \in \Sigma_a^3\},
\end{align*}
\item[2.] the set $I_1$ of initial states is
\[
\{ (i,l,(0,0,0),(0,0,0)) \in S \ : \ i \geq 4 \},
\]
\item[3.] the transition table $T_1$ contains the tuples $(s,(x,y),t)\in S_1 \times \Sigma_a^2 \times S_1$ that satisfy $w'=(w_2,w_3,y)$ and
one of the following conditions:
\begin{itemize}
\item[a.] $i\neq \xi, (j,l')=(i-1,l), A(P(i,l),v,x,w_1,v'),$
\item[b.] $i= \xi, l=1, (j,l')=(\nu,l), A(P(i,l),v,x,w_1,v')$,
\item[c.] $i= \xi, l=0, (j,l')=(i-1,l), A(P(i,l),v,x,w_1,v')$
\item[d.] $i=4, j=3$, $A(P(4,l),v,x,w_1,v'), B(a_3,a_2,a_1,v',w_2,w_3,y)$,
\end{itemize}
where $s=(i,l,v,w)$, $w=(w_1,w_2,w_3)$ and $t=(j,k,v',w')$,
\item[4.] the set $F_1$ of final states is $\{ (i,l,w,y) \in S_1 \ : \ i=3\}$.
\end{itemize}

\noindent We leave it to the reader to check the details that $\Cal A$ indeed recognizes the set $\{ z * z' \ : \ z,z' \in \Sigma_a^*, z \rightsquigarrow_1 z'\}$. The automata we constructed is non-deterministic, but as mentioned above there is deterministic finite automaton that recognizes the same set.

\subsection*{Automata for Algorithm 2 and 3} We now describe the non-deterministic automata $\Cal A_2$ and $\Cal A_3$ recognizing the sets $\{ z * z' \ : \ z,z' \in \Sigma_a^*, z \rightsquigarrow_2 z'\}$ and $\{ z * z' \ : \ z,z' \in \Sigma_a^*, z \rightsquigarrow_3 z'\}$. Again, we have to fix some notation first. Let $C\subseteq  \N_{\leq m}^3 \times \N_{\leq m}^3 \times \N_{\leq m}^3$ be the set of triples $(u,v,w)\in C$ such that
\[
w =\left\{
  \begin{array}{ll}
    (v_1+1,0,v_3-1), & \hbox{if $v_1 < u_1$, $v_2 = u_2$ and $v_3>0$;} \\
    (v_1,v_2,v_3), & \hbox{otherwise.}
  \end{array}
\right.
\]
The relation $C$ represents the operation performed in both Algorithm 2 and 3. As for $A$ and $B$ above, the values of the variable $u$ correspond to the relevant part of the continued fraction, while the values of the variables $v$ and $w$ represent the entries in the moving window, before and after any changes are carried out. For $i \in \{3,\dots,\nu\}$ and $l\in \{0,1\}$,
\[
Q(i,l) := \left\{
         \begin{array}{ll}
          (a_i,a_{i-1},a_{\nu}), & \hbox{$i=\xi+1$ and $l=1$;} \\
          (a_i,a_{\nu},a_{\nu-1}), & \hbox{$i=\xi$ and $l=1$;} \\
          (a_i,a_{i-1},a_{i-2}), & \hbox{otherwise.} \\
         \end{array}
       \right.
\]

\noindent We start with an informal description of the automaton $\Cal A_2$. Let $z=z_l\dots z_1\in \Sigma_a^*$ and suppose that $z'=z_l'\dots z_1'$ is the output of Algorithm 2 on input $z$. To perform the operations at step $k$ in Algorithm 2, we again need to know a certain part of the continued fraction expansion of $a$; in this case $(a_k,a_{k-1},a_{k-2})$. As before it is enough to know the natural numbers $i\leq \nu$ with $a_k=a_i$, and whether $k <\nu$. Set $l$ to be $1$ if $k> \nu$ and $0$ otherwise.
 Then $Q(i,l) = (a_k,a_{k-1},a_{k-2})$. When constructing $\Cal A_2$, we have to be careful: the Algorithm 2 runs from the right to the left, but the automaton reads the input from the left to the right. Let $(v_1',v_2') \in \Sigma_a^2$ be such that $(z_k,v_1',v_2')$ are the entries in the moving window before the changes at step $k$ are made. Then at step $k$, the entries change as follows:
 \[
(z_k,v_1',v_2')\mapsto (v_1,v_2,z_{k-2}'),
\]
for some pair $(v_1,v_2) \in \Sigma_a^2$ with $C(Q(i,l),z_k,v_1',v_2',v_1,v_2,z_{k-2}')$. So when the automaton reads in $(z_{k-2},z_{k-2}')$, the value of $z_k$ is used to determine $z_{k-2}'$. Hence in contrast to $\Cal A_1$, the automaton $\Cal A_2$ has to remember the value of $z_k$, and not the value of $z_k'$. We define the states of $\Cal A_2$ to be tuples $(i,l,v,w)\in \{0,\dots,m\}\times \{0,1\} \times \Sigma_a^2 \times \Sigma_a^2$. The pair $v$ is again used to save the entries of the moving window, and $w$ is needed to remember the previously read entries of $z$. The automaton moves from one of these states to another according to the rules described in Algorithm 2. However, since the automaton reads the input backwards, the automaton will go from a state $(i,l,v,w)$ to a state $(i',l',v',w')$ if $Q(i,l)$ and $Q(i',l')$ are the correct parts of the continued fraction expansion of $a$ and the algorithm transforms $(z_k,v_1',v_2')$ to $(v_1,v_2,z_{k-2}')$.
 \newline


\noindent Here is the definition of the automaton $\Cal A_2 =(S_2,I_2,T_2,F_2)$.
\begin{itemize}
\item [1.] The set $S_2$ of states of $\Cal A_2$ is
\begin{align*}
\{ (i,1,v,w) \ : \ \xi \leq i & \leq \nu, v,w \in \Sigma_a^2\}\\
&\cup \{ (i,0,v,w) \ : \ 2 \leq i\leq \xi, v,w \in \Sigma_a^2\},
\end{align*}
\item[2.] the set $I_2$ of initial states is
\[
\{ (i,l,(0,0,0),(0,0,0)) \in S \ : \ i \geq 3 \},
\]
\item[3.] the transition table $T_2$ contains the tuples $(s,(x,y),t)\in S_2 \times \Sigma_a^2 \times S_2$ that satisfy $w'=(w_2,x)$ and
one of the following conditions:
\begin{itemize}
\item[a.] $i\neq \xi, (j,l')=(i-1,l), C(Q(i,l),w_1,v',v,y),$
\item[b.] $i= \xi, l=1, (j,l')=(\nu,l), C(Q(i,l),w_1,v',v,y)$,
\item[c.] $i= \xi, l=0, (j,l')=(i-1,l), C(Q(i,l),w_1,v',v,y)$
\item[d.] $i=3, j=2$, $C(Q(i,0),w,x,v,y)$,
\end{itemize}
where $s=(i,l,v,w)$, $w=(w_1,w_2)$ and $t=(j,k,v',w')$,
\item[4.] the set $F_2$ of final states is $\{ (i,l,w,y) \in S_2 \ : \ i=3\}$.
\end{itemize}

\noindent As in the case of Algorithm 1, we leave it to the reader to verify that $\Cal A_2$ recognizes the set $\{ z * z' \ : \ z,z' \in \Sigma_a^*, z \rightsquigarrow_2 z'\}$. As before, while $\Cal A_2$ is non-deterministic,  there is a deterministic automata recognizing the same set as $\Cal A_2$.\\

\noindent It is left to construct the automaton for Algorithm 3. The only difference between Algorithm 2 and 3 is the direction in which the algorithm runs over the input. Hence the only adjustment we need to make to $\Cal A_2$, is to address the change in direction. Let $\Cal A_3 =(S_2,I_2,T_3,F_2)$ be the automaton that has the same states as $\Cal A_2$, but whose transition table $T_3$ contains the tuples $(s,(x,y),t)\in S_2 \times \Sigma_a^2 \times S_2$ that satisfy $w'=(w_2,y)$ and
one of the following conditions:
\begin{itemize}
\item[a.] $i\neq \xi, (j,l')=(i-1,l), C(Q(i,l),v,x,w_1,v'),$
\item[b.] $i= \xi, l=1, (j,l')=(\nu,l), C(Q(i,l),v,x,w_1,v')$,
\item[c.] $i= \xi, l=0, (j,l')=(i-1,l), C(Q(i,l),v,x,w_1,v')$
\item[d.] $i=3, j=2$, $C(Q(i,0),v,x,w,y)$,
\end{itemize}
where $s=(i,l,v,w)$, $w=(w_1,w_2)$ and $t=(j,k,v',w')$.\\

\noindent The set $\{ z * z' \ : \ z,z' \in \Sigma_a^*, z \rightsquigarrow_3 z'\}$ is recognized by $\Cal A_3$. So there is also a deterministic automaton recognizes this set. This completes the proof of Theorem A.

\bibliographystyle{plain}
  \bibliography{hieronymi}

\end{document}